\documentclass[11pt,reqno]{article}
\usepackage{latexsym, amsmath, amssymb, amsthm, a4, epsfig}
\usepackage{graphicx}

\newtheorem{theorem}{Theorem}[section]
\newtheorem{cor}[theorem]{Corollary}
\newtheorem{lemma}[theorem]{Lemma}
\newtheorem{prop}[theorem]{Proposition}

\newcommand{\nm}{\noalign{\smallskip}}
\newcommand{\ds}{\displaystyle}
\newcommand{\p}{\partial}

\newcommand{\eqnref}[1]{(\ref {#1})}

\newcommand{\Rbb}{\mathbb{R}}

\newcommand{\la}{\langle}
\newcommand{\ra}{\rangle}

\newcommand{\Hcal}{\mathcal{H}}

\newcommand{\Kcal}{\mathcal{K}}

\newcommand{\Scal}{\mathcal{S}}

\newcommand{\Ga}{\alpha}

\newcommand{\Gd}{\delta}

\newcommand{\Gvf}{\varphi}

\newcommand{\Gl}{\lambda}

\newcommand{\Gv}{\nu}
\newcommand{\Gp}{\pi}

\newcommand{\Gs}{\sigma}

\newcommand{\GG}{\Gamma}

\newcommand{\GO}{\Omega}

\newcommand{\beq}{\begin{equation}}
\newcommand{\eeq}{\end{equation}}

\numberwithin{equation}{section}
\numberwithin{figure}{section}

\begin{document}

\title{A concavity condition for existence of a negative Neumann-Poincar\'e eigenvalue in three dimensions\thanks{\footnotesize This work was supported by National Research Foundation of Korea through grants No. 2016R1A2B4011304 and 2017R1A4A1014735.}}

\author{Yong-Gwan Ji\thanks{Department of Mathematics and Institute of Applied Mathematics, Inha University, Incheon 22212, S. Korea (22151063@inha.edu, hbkang@inha.ac.kr).} \and Hyeonbae Kang\footnotemark[2] }

\maketitle

\begin{abstract}
It is proved that if a bounded domain in three dimensions satisfies a certain concavity condition, then the Neumann-Poincar\'e operator on the boundary of the domain or its inversion in a sphere has at least one negative eigenvalue. The concavity condition is quite simple, and is satisfied if there is a point on the boundary at which the Gaussian curvature is negative.
\end{abstract}

\noindent{\footnotesize {\bf AMS subject classifications}. 47A45 (primary), 31B25 (secondary)}

\noindent{\footnotesize {\bf Key words}. Neumann-Poincar\'e operator, negative eigenvalue, concavity condition, negative Gaussian curvature, inversion.}

\section{Introduction}

The Neumann-Poincar\'e (abbreviated by NP) operator is a boundary integral operator which appears naturally when solving classical boundary value problems using layer potentials. Its study goes back to C. Neumann \cite{Neumann-87} and Poincar\'e \cite{Poincare-AM-87} as the name of the operator suggests. Lately interest in the spectral properties of NP operator is growing rapidly, which is due to their relation to plasmonics \cite{ACKLM, AMRZ, MFZ}, and significant progress is being made among which are work on continuous spectrum in two dimensions \cite{BZ, HKL, KLY, PP2} and Weyl asymptotic of eigenvalues in three dimensions \cite{Miya} to name only a few.

However, there are still several puzzling questions on the NP spectrum (spectrum of the NP operator). A question on existence of negative NP eigenvalues in three dimensions is one of them. Unlike two-dimensional NP spectrum which is symmetric with respect to $0$ and hence there are the same number of negative eigenvalues as positive ones (see, for example \cite{HKL, KPS}), not so many surfaces (boundaries of three-dimensional domains) are known to have negative NP eigenvalues. In fact, NP eigenvalues on spheres are all positive, and it is only in \cite{Ahner} which was published in 1994 that the NP operator on a very thin oblate spheroid is shown to have a negative eigenvalue. We emphasize that the NP eigenvalues on ellipsoids can be found explicitly using Lam\'e functions for which we also refer to \cite{AA, FK, Mart, Ritt}. As far as we are aware of, there is no surface other than ellipsoids known to have a negative NP eigenvalue.

It is the purpose of this paper to present a simple geometric condition which guarantees existence of a negative NP eigenvalue. To present the condition, let $\GO$ be a bounded domain with the $C^{1,\Ga}$ boundary for some $\Ga>0$. We say $\p\GO$ is concave with respect to $p \in \GO$ if there is a point $x \in \p\GO$ such that
\beq\label{concave}
(x-p) \cdot \nu_x <0,
\eeq
where $\nu_x$ denotes the outward unit normal vector to $\p\GO$. We emphasize that if $\p\GO$ is $C^2$, then this condition is fulfilled if there is a point on $\p\GO$ where the Gaussian curvature is negative. In fact, if $(x-p) \cdot \nu_x \ge 0$ for all $p \in \GO$ and $x \in \p\GO$, then $\GO$ is convex and hence the Gaussian curvatures on $\p\GO$ are non-negative. We prove in this paper that if the concavity condition \eqnref{concave} holds for some $p \in \GO$, then the NP operator defined either on $\p\GO$ or the surface of inversion with respect $p$ has at least one negative eigenvalue (see Theorem \ref{mainthm1} and Corollary \ref{mainthm2}). We emphasize that \eqnref{concave} is not a necessary condition for existence of a negative NP eigenvalue; oblate spheroids have negative NP eigenvalues as mentioned before. 

This paper is organized as follows. We review in section \ref{sec:main} the definition of the NP operator and state main results of this paper. Section \ref{sec:proof} is to prove the transformation formula for the NP operator under the inversion in a sphere. We use this formula to prove the main results.

\section{The NP operator and statements of main results}\label{sec:main}

Let $\GG(x)$ be the fundamental solution to the
Laplacian, {\it i.e.},
\beq\label{gammacond}
\GG (x) =
\begin{cases}
\ds \frac{1}{2\pi} \ln |x|, \quad & d=2, \\ \nm \ds
-\frac{1}{4\pi |x|}, \quad & d = 3.
\end{cases}
\eeq
As before, let $\GO$ be a bounded domain with the $C^{1,\Ga}$ boundary for some $\Ga>0$. The single layer potential $\Scal_{\p \GO} [\Gvf]$ of a density function $\Gvf$ is defined by
\beq
\Scal_{\p \GO} [\Gvf] (x):= \int_{\p \GO} \GG (x-y) \Gvf (y) \, d\Gs(y), \quad x \in \Rbb^d.
\eeq
It is well known (see, for example, \cite{AmKa07Book2, Fo95}) that $\Scal_{\p \GO} [\Gvf]$ satisfies the jump relation
\beq\label{singlejump}
\frac{\p}{\p\nu} \Scal_{\p \GO} [\Gvf] \Big|_\pm (x) = \biggl( \pm \frac{1}{2} I + \Kcal_{\p \GO}^* \biggr) [\Gvf] (x),
\quad x \in \p \GO,
\eeq
where $\frac{\p}{\p\nu}$ denotes the outward normal derivative, the subscripts $\pm$ indicate the limit from outside and inside of $\GO$, respectively, and the operator $\Kcal_{\p \GO}^*$ is defined by
\beq\label{introkd}
\Kcal_{\p \GO}^* [\Gvf] (x):=
\int_{\p \GO} \nu_x \cdot \nabla_x \GG(x-y) \Gvf(y)\,d\Gs(y), \quad x \in \p  \GO.
\eeq

The operator $\Kcal_{\p \GO}^*$ is called the NP operator associated with the domain $\GO$ (or its boundary $\p\GO$). The operator $\Scal_{\p\GO}$, as an operator on $\p\GO$, maps $H^{-1/2}(\p\GO)$ into $H^{1/2}(\p\GO)$ continuously, and is invertible if $d=3$. If $d=2$, then there are domains where $\Scal_{\p\GO}$ has one-dimensional kernel, but by dilating the domain, it can be made to be invertible (see \cite{Verch-JFA-84}). So, the bilinear form $\la \cdot,\cdot \ra_{\p\GO}$, defined by
\beq\label{innerp}
\la \Gvf,\psi \ra_{\p\GO}:= -\la \Gvf,  \Scal_{\p \GO}[\psi] \ra,
\eeq
for $\Gvf, \psi \in H^{-1/2}(\p \GO)$, is actually an inner product on $H^{-1/2}(\p \GO)$ and it yields the norm equivalent to usual $H^{-1/2}$-norm (see, for example, \cite{KKLSY}). Here, $H^s$ denotes the usual $L^2$-Sobolev space of order $s$ and $\la \cdot, \cdot \ra$ is the $H^{-1/2}$-$H^{1/2}$ duality pairing. We denote the space $H^{-1/2}(\p \GO)$ equipped with the inner product $\la \cdot,\cdot \ra_{\p\GO}$ by $\Hcal^*$.

It is proved in \cite{KPS} that the NP operator $\Kcal_{\p \GO}^*$ is self-adjoint with respect to the inner product $\la \cdot,\cdot \ra_{\p\GO}$. In fact, it is an immediate consequence of
Plemelj's symmetrization principle
\beq\label{Plemelj}
\Scal_{\p \GO} \Kcal^*_{\p \GO} = \Kcal_{\p \GO} \Scal_{\p \GO}.
\eeq
Here $\Kcal_{\p \GO}$ is the adjoint of $\Kcal^*_{\p \GO}$ with respect to the usual $L^2$-inner product. Since $\Kcal_{\p \GO}^*$ is compact on $\Hcal^*$ if $\p\GO$ is $C^{1,\Ga}$, it has eigenvalues converging to $0$.

For a fixed $r>0$ and $p \in \Rbb^d$, let $T_p : \Rbb^d \setminus \{p\} \rightarrow \Rbb^d \setminus \{p\}$, $d =2,3$, be the inversion in a sphere, namely,
\beq
T_p x := \frac{r^2}{|x-p|^2} (x-p) + p.
\eeq
For a given bounded domain $\GO$ in $\Rbb^d$, let $\p \GO_p^*$ be the inversion of $\p \GO$, {\it i.e.}, $\p \GO_p^*= T_p(\p \GO)$. If we invert $\GO$ in spheres of two different radii, then the resulting domains are dilations of each other. Since NP spectrum is invariant under dilation as one can see easily by a change of variables, we may fix the radius of the inversion sphere once for all.

The following is the main results of this paper.
\begin{theorem}\label{mainthm1}
Let $\GO$ be a bounded domain whose boundary is $C^{1,\Ga}$ smooth for some $\Ga>0$. If the concavity condition \eqnref{concave} holds for some $p \in \GO$, then either $\Kcal_{\p\GO}^*$ or $\Kcal_{\p\GO_p^*}^*$ has a negative eigenvalue.
\end{theorem}

We have the following corollary as an immediate consequence of Theorem \ref{mainthm1}.

\begin{cor}\label{mainthm2}
Suppose $\p\GO$ is $C^{2}$ smooth. If there is a point on $\p\GO$ where the Gaussian curvature is negative, then either $\Kcal_{\p\GO}^*$ or $\Kcal_{\p\GO_p^*}^*$ for some $p \in \GO$ has a negative eigenvalue.
\end{cor}

\section{Inversion in a sphere}\label{sec:proof}

Just for simplicity we now assume the center $p$ of the inversion sphere is $0$, and denote $\p\GO_p^*$ and $T_p$ by $\p\GO^*$ and $T$, respectively.

Let $x^*:= Tx$. Since
$$
\frac{\p x_i}{\p x^*_j} = \frac{r^2 \Gd_{ij}}{|x^*|^2} - 2r^2 \frac{x^*_i x^*_j}{|x^*|^4},
$$
the Jacobian matrix of $T^{-1}$ is given by
\beq\label{Jacobian}
J_{T^{-1}} = \frac{r^2}{|x^*|^2} \Big(I - 2\frac{x^*}{|x^*|}\frac{(x^*)^t}{|x^*|} \Big) = \frac{|x|^2}{r^2} \Big(I - 2\frac{x}{|x|}\frac{x^t}{|x|} \Big).
\eeq
Here $t$ denotes transpose. So, $T$ is conformal. The change of variable formulas for line and surface are respectively given as follows:
\begin{align}
& ds(x^*) = \frac{r^2}{|x|^2} ds(x), \label{line}\\
& dS(x^*) = \frac{r^4}{|x|^4} dS(x). \label{surface}
\end{align}

It is known (see \cite{MacMillan-book}) that
\begin{align}
|x^*-y^*| = \frac{r^2}{|x||y|}  |x-y|.
\end{align}
So we have
\beq\label{tildefund}
\GG(x^*-y^*) = \begin{cases}  \GG(x-y) -\GG(x) -\GG(y) + \GG(r^2)  &\text{if  } d = 2, \\
\nm
\ds \frac{|x||y|}{r^2} \GG(x-y) &\text{if  } d = 3. \end{cases}
\eeq
For a function $\Gvf$ defined on $\p\GO$, define $\Gvf^*$ on $\p\GO^*_p$ by
\beq\label{inversionfunc}
\Gvf^*(y^*) := \Gvf (y) \frac{|y|^d}{r^d}.
\eeq
Then, we can easily see using \eqref{tildefund} that the following relation between the single layer potentials on domains $\p \GO$ and $\p\GO^*$ holds (see also \cite{MacMillan-book}):
\beq\label{singleform}
\Scal_{\p\GO^*} [ \Gvf^* ](x^*)  =
\begin{cases} \Scal_{\p \GO} [\Gvf](x) - \Scal_{\p \GO} [\Gvf](0) + \Big( \ds \int_{\p\GO} \Gvf ds \Big) \big(\GG(r^2) - \GG(x)\big)  &\text{if  } d = 2, \\
\ds \frac{|x|}{r} \Scal_{{\p \GO}} [ \Gvf ](x) &\text{if  } d = 3 .
\end{cases}
\eeq

We note that the map $\Gvf \mapsto \Gvf^*$ is a conformal map from $\Hcal^*(\p \GO)$ to $\Hcal^*(\p \GO^*)$, in three dimensions, that is,
\beq
\la \Gvf^*, \psi^* \ra_{\p\GO^*} = \la \Gvf, \psi \ra_{\p\GO}.
\eeq
This relation is also true in two dimensions if $\Gvf$ and $\psi$ are of mean zero.
The relationship between outward unit normal vectors $\Gv_{x^*}$ on $\p\GO^*$ and $\Gv_x$ on ${\p \GO}$ are given as follows:
\beq\label{normalform}
\Gv_{x^*} = (-1)^m \Big( I - 2\frac{x}{|x|} \frac{x^t}{|x|} \Big) \Gv_x,
\eeq
where $m = 1$ if $0$ is an exterior point of $\GO$ and $m= 0$ if $0$ is an interior point of $\GO$. We emphasize that $0$ is the inversion center.

The NP operators $ \Kcal^{*}_{\p \GO}$ and $\Kcal^{*}_{\p\GO^*}$ are related in the following way:

\begin{lemma} \label{tildNP}
Suppose that $0$ is the center of the inversion sphere.
\begin{itemize}
\item[(i)] If $0 \in \GO$, then
\beq \label{tildNPinterior}
\Kcal^{*}_{\p\GO^*} [ \Gvf^* ](x^*) = \begin{cases} \displaystyle -\frac{|x|^2}{r^2} \Kcal^{*}_{\p \GO} [ \Gvf ](x) + \Big( \int_{\p \GO} \Gvf ds \Big) \frac{ x \cdot \Gv_x }{2\Gp r^2}  \quad &\text{if  } d = 2, \\
\nm
\ds -\frac{|x|^3}{r^3} \Kcal^{*}_{\p \GO} [ \Gvf ](x) - \frac{|x|(x \cdot \Gv_x)}{r^3} \Scal_{{\p \GO}} [ \Gvf ](x)  \quad &\text{if  } d = 3. \end{cases}
\eeq

\item[(ii)] If $0 \in \overline{\GO}^c$, then
\beq \label{tildNPexterior}
\Kcal^{*}_{\p\GO^*} [ \Gvf^* ](x^*) =
\begin{cases} \displaystyle \frac{|x|^2}{r^2} \Kcal^{*}_{\p \GO} [ \Gvf ](x) - \Big( \int_{\p \GO} \Gvf ds \Big) \frac{x \cdot \Gv_x}{2\Gp r^2}  \quad &\text{if  } d = 2, \\
\nm
\ds \frac{|x|^3}{r^3} \Kcal^{*}_{\p \GO} [ \Gvf ](x) + \frac{|x|(x \cdot \Gv_x)}{r^3} \Scal_{{\p \GO}} [ \Gvf ](x)  \quad &\text{if  } d = 3.
\end{cases}
\eeq

\end{itemize}
\end{lemma}

\begin{proof}
Since the difference of proofs for \eqnref{tildNPinterior} and \eqnref{tildNPexterior} is just the sign of the normal vector, we only prove the first one.

If $d=2$, we use \eqnref{Jacobian}, \eqnref{line}, \eqnref{tildefund}, \eqnref{inversionfunc}, and \eqnref{normalform} to have
\begin{align*}
\Kcal^{*}_{\p\GO^*} [ \Gvf^* ](x^*) &= \int_{\p\GO^*} \Gv_{x^*} \cdot \nabla_{x^*} \GG(x^* - y^*)  \Gvf^*(y^*) \; ds(y^*) \\
&= \int_{{\p \GO}} - \Big( I - 2\frac{x}{|x|} \frac{x^t}{|x|} \Big) \Gv_x \cdot J_{T^{-1}}^t \nabla_{x} \Big( \GG(x-y) -\GG(x) \Big) \Gvf (y) \; ds(y) \\
&= \int_{{\p \GO}} - \Gv_x \cdot \Big( I - 2\frac{x}{|x|} \frac{x^t}{|x|} \Big)^t J_{T^{-1}}^t \nabla_{x} \Big( \GG(x-y) -\GG(x) \Big)  \Gvf (y) \;ds(y) \\
&= - \frac{|x|^2}{r^2} \int_{{\p \GO}} \Gv_x \cdot \Big( \nabla_x \GG(x-y) - \nabla_x \GG(x)   \Big)  \Gvf (y) \; ds(y) \\
&= -\frac{|x|^2}{r^2} \Kcal^{*}_{\p \GO} [ \Gvf ](x) + \Big( \int_{\p \GO} \Gvf ds \Big) \frac{x \cdot \Gv_x}{2\Gp r^2}.
\end{align*}

The three-dimensional case can be proved similarly. In fact, we have
\begin{align*}
\Kcal^{*}_{\p\GO^*} [ \Gvf^* ](x^*) &= \int_{\p\GO^*} \Gv_{x^*} \cdot \nabla_{x^*} \GG(x^* - y^*) \; \Gvf^*(y^*) \;  dS(y^*) \\
&= \int_{{\p \GO}} - \Big( I - 2\frac{x}{|x|} \frac{x^t}{|x|} \Big) \Gv_x \cdot J_{T^{-1}}^t \nabla_{x} \Big( \frac{|x||y|}{r^2} \GG(x-y) \Big) \Gvf (y)  \; \frac{r}{|y|}dS(y) \\
&= \int_{{\p \GO}} - \Gv_x \cdot \Big( I - 2\frac{x}{|x|} \frac{x^t}{|x|} \Big)^t J_{T^{-1}}^t \nabla_{x} \Big( \frac{|x||y|}{r^2} \GG(x-y) \Big) \Gvf (y)  \; \frac{r}{|y|}dS(y) \\
&= - \frac{|x|^2}{r} \int_{{\p \GO}} \Gv_x \cdot \Big( \frac{|x|}{r^2} \nabla_x \GG(x-y) + \frac{x}{r^2|x|}\GG(x-y)   \Big) \Gvf (y) \; dS(y) \\
&= \displaystyle -\frac{|x|^3}{r^3} \Kcal^{*}_{\p \GO} [ \Gvf ](x) - \frac{|x|(x \cdot \Gv_x)}{r^3} \Scal_{{\p \GO}} [ \Gvf ](x).
\end{align*}
This completes the proof.
\end{proof}

If $\Gvf$ is an eigenvector of $\Kcal^{*}_{\p \GO}$ with corresponding eigenvalue $\Gl \neq 1/2$, then $\int_{\p \GO} \Gvf ds=0$. So \eqnref{inversionfunc} and Lemma \ref{tildNP} for $d=2$ show that $\Gvf^*$ is an eigenvector of $\Kcal^{*}_{\p\GO^*}$ and the corresponding eigenvalue is $-\Gl$ if $0\in \GO$ and $\Gl$ if $0 \in \overline{\GO}^c$. Since the spectrum $\Gs(\Kcal^{*}_{\p \GO})$ of the NP operator in two dimensions is symmetric with respect to $0$, we infer that $\Gs(\Kcal^{*}_{\p\GO^*}) = \Gs(\Kcal^{*}_{\p \GO})$, namely, the spectrum is invariant under inversion. This fact is known \cite{Schi}, but the above yields an alternative proof.

In three dimensions, we obtain the following identities.

\begin{prop}\label{tildeNPinnerprod}
Suppose $d=3$ and $0$ is the center of the inversion sphere.
\begin{itemize}
\item[(i)] If $0 \in \GO$, then
\beq\label{inverform1}
\la \Kcal^{*}_{\p\GO^*} [ \Gvf^* ], \Gvf^* \ra_{\p\GO^*} + \la \Kcal^{*}_{\p \GO} [ \Gvf ], \Gvf \ra_{\p\GO} = \int_{\p \GO}  \frac{x \cdot \Gv_x}{|x|^2} | \Scal_{{\p \GO}} [ \Gvf ](x)|^2 \; dS.
\eeq

\item[(ii)] If $0 \in \overline{\GO}^c$, then
\beq\label{inverform2}
\la \Kcal^{*}_{\p\GO^*} [ \Gvf^* ], \Gvf^* \ra_{\p\GO^*} - \la \Kcal^{*}_{\p \GO} [ \Gvf ], \Gvf \ra_{\p\GO} = -\int_{\p \GO}  \frac{x \cdot \Gv_x}{|x|^2} | \Scal_{{\p \GO}} [ \Gvf ](x)|^2 \; dS.
\eeq

\end{itemize}
\end{prop}

\begin{proof}
The identity follows from \eqnref{singleform} and \eqnref{tildNPinterior} immediately. In fact, we have
\begin{align*}
\la \Kcal^{*}_{\p\GO^*} [ \Gvf^* ], \Gvf^* \ra_{\p\GO^*}
&= -\int_{\p\GO^*}  \Kcal^{*}_{\p\GO^*} [ \Gvf^* ](x^*) \; \overline{\Scal_{\p\GO^*} [ \Gvf^* ](x^*)}  dS(x^*) \\
&= \int_{\p \GO}  \Big( \frac{|x|^3}{r^3} \Kcal^{*}_{\p \GO} [ \Gvf ](x) + \frac{|x|(x \cdot \Gv_x)}{r^3} \Scal_{{\p \GO}} [ \Gvf ](x) \Big) \overline{\Scal_{{\p \GO}} [ \Gvf ](x)}  \frac{r^3}{|x|^3} dS(x) \\
&= -\la \Kcal^{*}_{\p \GO} [ \Gvf ], \Gvf \ra_* + \int_{\p \GO}  \frac{x \cdot \Gv_x}{|x|^2} \; | \Scal_{{\p \GO}} [ \Gvf ](x) |^2 dS(x),
\end{align*}
which proves \eqnref{inverform1}. \eqnref{inverform2} can be proved similarly.
\end{proof}

\medskip
We are now ready to prove Theorem \ref{mainthm1}.

\noindent{\sl Proof of Theorem \ref{mainthm1}}.
Suppose $\GO$ satisfies \eqnref{concave} at $p \in \GO$. Without loss of generality we assume that $p=0$. Then there is $x_0 \in \p\GO$ such that $x_0 \cdot \nu_{x_0} <0$. Choose an open neighborhood $U$ of $x_0$ in $\p \GO$ so that
$x \cdot \nu_x <0$ for all $x \in U$. Since $\Scal_{\p\GO}: H^{-1/2}(\p\GO) \to H^{1/2}(\p\GO)$ is invertible in three dimensions (see \cite{Verch-JFA-84}), we may choose $\Gvf \neq 0$ so that $\Scal_{\p\GO}[\Gvf]$ is supported in $U$. We then infer from \eqnref{inverform1} that
$$
\la \Kcal^{*}_{\p\GO^*} [ \Gvf^* ], \Gvf^* \ra_{\p\GO^*} + \la \Kcal^{*}_{\p \GO} [ \Gvf ], \Gvf \ra_{\p\GO} <0.
$$
Therefore, the numerical range of either $\Kcal^{*}_{\p\GO^*}$ or $\Kcal^{*}_{\p \GO}$ has a negative element. It implies that either $\Kcal^{*}_{\p\GO^*}$ or $\Kcal^{*}_{\p \GO}$ has at least one negative eigenvalue. This completes the proof. \qed

\section*{Acknowledgement}

We thank Yoshihisa Miyanishi for many inspiring discussions.


\end{document}